\documentclass[11pt]{amsart}

\usepackage[T1]{fontenc}
\usepackage[utf8]{inputenc}
\usepackage{lmodern}
\usepackage{amsmath,amssymb,amsthm,mathtools}
\usepackage{mathrsfs}
\usepackage{bm}
\usepackage{microtype}
\usepackage[hidelinks]{hyperref}
\usepackage[nameinlink,noabbrev]{cleveref}
\usepackage[backend=biber,style=numeric,sorting=nyt]{biblatex}
\numberwithin{equation}{section}

\newtheorem{theorem}{Theorem}[section]
\newtheorem{proposition}[theorem]{Proposition}
\newtheorem{lemma}[theorem]{Lemma}
\newtheorem{corollary}[theorem]{Corollary}
\theoremstyle{definition}

\title[Hessian cancellation and quartic nondegeneracy]{Exact Hessian Cancellation and Quartic Nondegeneracy at the Round Circle}
\author{Aya Ishizeki}
\thanks{Corresponding author: Aya Ishizeki. Email: \href{mailto:aishizek@rimath.saitama-u.ac.jp}{aishizek@rimath.saitama-u.ac.jp}. ORCID: \href{https://orcid.org/0009-0008-7250-575X}{0009-0008-7250-575X}}
\address{Department of Mathematics, Faculty of Science, Saitama University, 255 Shimo-Okubo, Sakura-ku, Saitama City, Saitama 338-8570, Japan}
\email{aishizek@rimath.saitama-u.ac.jp}

\begin{document}

\begin{abstract}
We study the decomposed M\"obius energies \(E_1\) and \(E_2\) at the round circle and determine their complete normal Hessians:
\[
H_1(C)=2\pi\bigl(|D_s|^3-|D_s|\bigr)I_2,
\qquad
H_2(C)=-\frac{4\pi}{3}\bigl(|D_s|^3-|D_s|\bigr)I_2.
\]
Thus
\[
H_1(C)=3H_E(C),
\qquad
H_2(C)=-2H_E(C),
\]
where \(H_E(C)\) is the Hessian of the full M\"obius energy. This exact proportionality singles out the M\"obius-invariant combination
\[
\mathcal F=2E_1+3E_2+2\pi^2.
\]
We prove that \(\mathcal F(C)=0\) and that its first three derivatives vanish on the full space of parametrized variations.

For scalar binormal variations, the fourth variation has a Fourier tensor with complete \(3+1\) resonance cancellation and finite \(2+2\) fixed-sum blocks. Their positivity follows from an explicit weighted-difference tridiagonalization and a closed determinant formula. For arbitrary normal variations in \(\mathbb R^3\), M\"obius invariance in \(\mathbb R^4\) reduces the radial component to a second binormal direction. The resulting two-component quartic form splits into trace, symmetric-traceless, and antisymmetric blocks, which are positive on the active coordinates. Consequently,
\[
D^4\mathcal F(C)[u,u,u,u]\ge0
\]
for every smooth real normal field \(u\), with equality exactly on the normal tangent space of the M\"obius family of round circles. This yields a fourth-order nondegeneracy phenomenon that is not determined by the total-energy Hessian alone.
\end{abstract}

\maketitle

\section{Introduction}

The M\"obius energy is a conformally invariant knot energy introduced to measure the geometric complexity of closed curves \cite{OHara1991,FreedmanHeWang1994}. Its variational structure has been studied extensively, both from the point of view of critical-point regularity and in connection with geometric gradient flows.

For a closed curve \(f\), the M\"obius energy admits the M\"obius-invariant decomposition
\[
E(f)=E_1(f)+E_2(f)+4
\]
introduced by Ishizeki and Nagasawa \cite{IshizekiNagasawa2014}. The two components \(E_1\) and \(E_2\) possess separate variational structures and are themselves M\"obius invariant. Their first and second variation formulae, \(L^2\)-gradient representations, and principal third-order parts are known. In particular, the leading terms in the component Hessians are
\[
2\pi(-\Delta_s)^{3/2}
\qquad\text{and}\qquad
-\frac{4\pi}{3}(-\Delta_s)^{3/2},
\]
respectively.

The purpose of the present paper is to investigate what additional structure becomes visible at the round circle when the two decomposed energies are treated separately.

Let
\[
C(s)=(\cos s,\sin s,0),
\qquad
s\in\mathbb T=\mathbb R/2\pi\mathbb Z,
\]
and let
\[
b=(0,0,1).
\]
We write a normal variation along \(C\) as
\[
u(s)=a(s)C(s)+c(s)b.
\]

The high symmetry of the circle allows the full lower-order part of the Hessian to be computed, not only its principal symbol. We obtain
\[
H_1(C)=2\pi\bigl(|D_s|^3-|D_s|\bigr)I_2,
\]
\[
H_2(C)=-\frac{4\pi}{3}\bigl(|D_s|^3-|D_s|\bigr)I_2.
\]
Thus the complete Hessians satisfy
\[
H_1(C)=3H_E(C),
\qquad
H_2(C)=-2H_E(C).
\]

This is not merely a relation between principal symbols: the complete lower-order corrections occur in exactly the same proportion.

This exact cancellation naturally singles out
\begin{equation}\label{eq:F-def}
\mathcal F=2E_1+3E_2+2\pi^2.
\end{equation}
The constant is chosen so that \(\mathcal F(C)=0\). Since the circle is stationary for both decomposed energies, the first variation vanishes, while the Hessian identities yield
\[
D^2\mathcal F(C)=0.
\]

A further cancellation occurs at third order. Using a M\"obius transformation that fixes the circle pointwise and acts by \(-I\) on its normal bundle, together with reparametrization invariance, we show that
\[
D^3\mathcal F(C)=0
\]
on the full space of parametrized variations. Thus
\[
D^j\mathcal F(C)=0,
\qquad
j=0,1,2,3.
\]
The fourth derivative is therefore the first potentially nontrivial jet of \(\mathcal F\) at \(C\).

The main result of this paper is that this fourth derivative is nonnegative on the full normal space and strictly positive away from the infinitesimal M\"obius directions of round circles.

\subsection{Related work}

The second variation of the full M\"obius energy was computed by He \cite{He2000}, who identified its principal third-order part as a multiple of the fractional Laplacian
\[
(-\Delta_s)^{3/2}.
\]
This order-three structure has subsequently played an important role in the analysis of the M\"obius gradient flow and in the regularity theory of critical knots; see, among others, the work of Blatt, Blatt--Reiter--Schikorra, and Blatt--Vorderobermeier \cite{Blatt2012,BlattReiterSchikorra2016,BlattVorderobermeier2019}. Reiter and Schumacher later exploited the same analytic scale in the construction of Sobolev gradients for numerical optimization of the M\"obius energy \cite{ReiterSchumacher2021}.

The results cited above provide the general variational framework and the third-order principal structure. Here we compute the complete normal Fourier multipliers of both decomposed energies at the round circle and use their exact proportionality to study a fourth-order degeneracy.

A complementary line of development began with the M\"obius-invariant decomposition of Ishizeki and Nagasawa. They derived variational formulae for \(E_1\) and \(E_2\) separately and later obtained \(L^2\)-representations of their gradients \cite{IshizekiNagasawa2015,IshizekiNagasawa2016}. Comparison estimates between the decomposed energies and the total M\"obius energy were also established in subsequent work \cite{IshizekiNagasawa2021}. Gunji subsequently obtained \(L^2\)-representations of the component Hessians, identified their principal third-order operators, and proved \L{}ojasiewicz--Simon gradient inequalities for the individual components \cite{Gunji2019}.

Fourier analysis of the Hessian at the round circle was considered more directly by Biesenbach in her 2021 dissertation \cite{Biesenbach2021}. Using an explicit second-variation representation and an ambient-coordinate Fourier expansion, she identified a nontrivial subset of the kernel of the Hessian of the full M\"obius energy and developed a possible route toward a complete classification. A complete kernel classification was not obtained in that dissertation.

In the present work, we instead use the moving normal frame of the circle and compute the two decomposed Hessians separately. Theorem~\ref{thm:complete-hessians} identifies their complete normal Fourier multipliers and their common kernel, sharpening the principal-order information at the round circle. Their exact proportionality is not determined by the total-energy Hessian alone and is the starting point of the higher-order phenomenon studied here: it yields the distinguished combination \(\mathcal F\) in \eqref{eq:F-def}, whose derivatives vanish through third order.

To the best of our knowledge, the resulting third-order degeneracy, the quartic Fourier tensor, and the full-normal quartic nondegeneracy proved below have not previously been established for the decomposed M\"obius energies.

\subsection{Main results}

We use
\[
|D_s|=(-\partial_s^2)^{1/2},
\qquad
|D_s|e^{iks}=|k|e^{iks}.
\]

Define
\[
\mathcal K
=
\left\{
aC+cb:\;
a,c\in\operatorname{span}_{\mathbb R}\{1,\cos s,\sin s\}
\right\}.
\]
This is the normal tangent space at \(C\) to the M\"obius family of unparametrized round circles.

We identify a normal field \(u=aC+cb\) with its coefficient pair \((a,c)\) in the moving normal frame \((C,b)\). The operators below act on this pair, with the unnormalized \(L^2\) pairing
\[
\langle(a,c),(\widetilde a,\widetilde c)\rangle_{L^2}
=
\int_0^{2\pi}(a\widetilde a+c\widetilde c)\,ds.
\]

\begin{theorem}[Complete decomposed Hessians at the round circle]\label{thm:complete-hessians}

For every smooth normal field \(u=aC+cb\),
\begin{align}
H_1(C)&=2\pi\bigl(|D_s|^3-|D_s|\bigr)I_2,\label{eq:complete-hessians}\\
H_2(C)&=-\frac{4\pi}{3}\bigl(|D_s|^3-|D_s|\bigr)I_2.\notag
\end{align}
Consequently,
\[
H_E(C)=\frac{2\pi}{3}\bigl(|D_s|^3-|D_s|\bigr)I_2.
\]
In particular,
\[
H_1(C)=3H_E(C),
\qquad
H_2(C)=-2H_E(C).
\]
Moreover,
\[
\ker H_1(C)=\ker H_2(C)=\ker H_E(C)=\mathcal K
\]
on the normal space.

\end{theorem}

\begin{theorem}[Vanishing through third order]\label{thm:lower-jets}

Let
\[
\mathcal F=2E_1+3E_2+2\pi^2.
\]
Then
\[
\mathcal F(C)=0,
\]
and
\[
D\mathcal F(C)=D^2\mathcal F(C)=D^3\mathcal F(C)=0
\]
on the full space of smooth parametrized variations.

\end{theorem}

\begin{theorem}[Quartic nondegeneracy]\label{thm:quartic}

For every smooth real normal field \(u=aC+cb\) along the unit circle in \(\mathbb R^3\),
\[
D^4\mathcal F(C)[u,u,u,u]\ge0.
\]
Moreover,
\[
D^4\mathcal F(C)[u,u,u,u]=0
\iff
u\in\mathcal K.
\]

Thus the fourth variation is strictly positive on every normal direction outside the infinitesimal M\"obius directions of round circles.

\end{theorem}

\subsection{Structure of the proof}

After identifying the vanishing lower-order jets, the proof of quartic nondegeneracy proceeds in three stages.

First, for purely binormal variations \(u=cb\), we derive an explicit physical-space formula for the quartic Taylor coefficient. Its Fourier polarization exhibits a special resonance structure: every \(3+1\) resonance vanishes identically, while the surviving \(2+2\) interactions are organized by a fixed total frequency \(S\). For each \(S\), the scalar quartic form reduces to a finite symmetric matrix \(M^{(S)}\). A weighted three-point difference transformation tridiagonalizes its leading principal blocks and leads to an explicit determinant formula. Sylvester's criterion then yields positivity.

Second, we pass from scalar binormal variations to arbitrary normal variations. We embed
\[
\mathbb R^3\subset\mathbb R^4
\]
and use a M\"obius transformation fixing the circle which sends the radial direction \(C\) to an additional binormal direction \(e_4\).

The vanishing of the lower-order jets is essential at this stage. Since
\[
D^j\mathcal F_{\mathbb R^4}(C)=0,
\qquad
j=1,2,3,
\]
the nonlinear correction terms in the fourth-order chain rule disappear. Thus, at fourth order, M\"obius invariance acts through the linearized transformation, reducing the full normal problem in \(\mathbb R^3\) to a two-component vector-valued binormal problem.

Finally, the resulting fixed-sum matrix variables decompose into
\[
\text{trace}\oplus\text{symmetric-traceless}\oplus\text{antisymmetric}
\]
parts. The corresponding three scalar matrix families are shown to be positive by three different mechanisms: a cumulative-sum factorization, a Green-matrix structure, and a weighted-difference tridiagonalization with positive row sums.

\subsection{Geometric meaning and scope}

Writing \(\delta E=E-4\) and \(\delta E_1=E_1-2\pi^2\), we have
\[
\mathcal F=3\delta E-\delta E_1.
\]
Theorem~\ref{thm:quartic} therefore identifies the first nonzero discrepancy between two M\"obius-invariant energies whose jets agree through order three, and shows that its quartic term has exactly the infinitesimal round-circle directions as its zero set. The result concerns the quartic jet: it does not assert a uniform nonlinear local-minimality or gradient-flow theorem. Section~8 explains the obstruction to quartic coercivity in \(H^{3/2}\) and the additional estimates required for such conclusions.

\section{Preliminaries}

Throughout the paper we work first in the smooth category. Let
\[
\mathbb T=\mathbb R/2\pi\mathbb Z
\]
and let
\[
f:\mathbb T\to\mathbb R^n
\]
be a smooth embedded closed curve with \(f'(s)\neq0\).

For the unit circle we use
\[
C(s)=(\cos s,\sin s,0),
\]
\[
\tau(s)=C'(s)=(-\sin s,\cos s,0),
\qquad
b=(0,0,1).
\]
We use the outward radial vector \(C(s)\), so every normal field is
\[
u=aC+cb.
\]

We write
\[
|D_s|=(-\partial_s^2)^{1/2},
\qquad
|D_s|e^{iks}=|k|e^{iks}.
\]

\subsection{Variational formulas and elementary variations}

Let
\[
d=f(s)-f(t).
\]
Then
\[
E_1(f)
=
\iint_{\mathbb T^2}
\frac{|f'(s)||f'(t)|-f'(s)\cdot f'(t)}
{|f(s)-f(t)|^2}\,ds\,dt,
\]
and
\[
E_2(f)
=
2\iint_{\mathbb T^2}
\left[
\frac{f'(s)\cdot f'(t)}{|d|^2}
-
\frac{(f'(s)\cdot d)(f'(t)\cdot d)}
{|d|^4}
\right]ds\,dt.
\]

For the circle,
\[
E_1(C)=2\pi^2,\qquad
E_2(C)=-2\pi^2,\qquad
E(C)=4.
\]

For later use, we also record the elementary first-variation formulas. For
\[
f_\varepsilon=f+\varepsilon u
\]
and an arc-length parametrized base curve,
\[
\delta |f'|=\tau\cdot u',
\]
\[
\delta(d\ell)=(\tau\cdot u')\,ds,
\]
and
\[
\delta\tau
=
u'-(\tau\cdot u')\tau
=
P_\tau^\perp u'.
\]

\subsection{Invariances, circle symmetries, and stationarity}

Reparametrization invariance gives the following identities. If \(G\) is reparametrization invariant, then
\[
DG(f)[\xi f']=0.
\]
Differentiating,
\[
D^2G(f)[u,\xi f']+DG(f)[\xi u']=0.
\]
Hence at a stationary point \(C\),
\[
D^2G(C)[u,\xi\tau]=0.
\]
Differentiating again, if \(DG(C)=D^2G(C)=0\), then
\[
D^3G(C)[u,v,\xi\tau]=0.
\]

For M\"obius invariance, we will repeatedly use the following chain-rule principle. Let \(G\) be \(C^m\), M\"obius invariant, and let \(T(C)=C\). If
\[
D^jG(C)=0,\qquad 1\le j\le m-1,
\]
then
\[
D^mG(C)[u_1,\ldots,u_m]
=
D^mG(C)[DT_Cu_1,\ldots,DT_Cu_m].
\]

This follows from the Faà di Bruno formula, since all terms containing \(D^jG(C)\) with \(j<m\) vanish.

The relevant symmetries of the round circle are as follows. Reflection
\[
R(x_1,x_2,x_3)=(x_1,x_2,-x_3)
\]
fixes \(C\), preserves \(C\) and \(\tau\), and reverses \(b\).

The M\"obius map
\[
T(x)=b+\frac{2(x-b)}{|x-b|^2}
\]
satisfies
\[
T(C)=C,
\]
\[
DT_C[C]=b,\qquad
DT_C[b]=C,\qquad
DT_C[\tau]=\tau.
\]

The normal-reversing map
\[
J(x)=\frac{(x_1,x_2,-x_3)}{|x|^2}
\]
fixes \(C\) and satisfies
\[
DJ_C[\tau]=\tau,\qquad
DJ_C[C]=-C,\qquad
DJ_C[b]=-b.
\]

These symmetries imply stationarity. Tangential first variations vanish by reparametrization invariance. Binormal first variations vanish by reflection. Radial first variations vanish by the map \(T\), since
\[
DT_C[C]=b.
\]
Therefore
\[
DE_i(C)=0,\qquad i=1,2,
\]
and consequently
\[
D\mathcal F(C)=0.
\]

\section{Complete Hessians at the Round Circle}

Let
\[
f_\varepsilon(s)=C(s)+\varepsilon c(s)b.
\]
Set
\[
h=s-t,\qquad
D=2-2\cos h,
\]
\[
p=c'(s),\qquad q=c'(t),\qquad z=c(s)-c(t).
\]

\subsection{Binormal second variations}

We have
\[
|f_\varepsilon'(s)|
=
\sqrt{1+\varepsilon^2p^2},
\]
\[
f_\varepsilon'(s)\cdot f_\varepsilon'(t)
=
\cos h+\varepsilon^2pq,
\]
and
\[
|f_\varepsilon(s)-f_\varepsilon(t)|^2
=
D+\varepsilon^2z^2.
\]
Expanding to order \(\varepsilon^2\),
\[
D^2E_1(C)[cb,cb]
=
\iint
\frac{(p-q)^2-z^2}{D}\,ds\,dt.
\]

For the second component, using
\[
f_\varepsilon'(s)\cdot
(f_\varepsilon(s)-f_\varepsilon(t))
=
\sin h+\varepsilon^2pz
\]
and the analogous formula at \(t\), one obtains
\[
D^2E_2(C)[cb,cb]
=
4\iint
\left[
\frac{pq}{D}
-
\frac{\sin h(p+q)z}{D^2}
+
\frac{z^2}{D^2}
\right]ds\,dt.
\]
The three displayed terms are interpreted as one combined integrable expression.

\subsection{Fourier spectra}

For
\[
A_k=
\int_{-\pi}^{\pi}
\frac{1-\cos(kh)}
{2-2\cos h}\,dh,
\]
the Fejér-kernel identity gives
\[
A_k=\pi|k|.
\]

Applying this identity to the first component, for \(c_k(s)=\cos(ks)\) with \(k\ge1\),
\[
D^2E_1(C)[c_kb,c_kb]
=
2\pi^2k(k^2-1),
\]
and since, for \(k\ge1\),
\[
\|c_k\|_{L^2}^2=\pi,
\]
the eigenvalue is
\[
\lambda_k^{(1)}=2\pi|k|(k^2-1).
\]

For the second component, for \(c_k=\cos(ks)\) with \(k\ge1\), define
\[
B_k
=
\int_{-\pi}^{\pi}
\frac{1-\cos(kh)-k^2(1-\cos h)}
{D^2}\,dh.
\]
Then
\[
B_{k+1}-2B_k+B_{k-1}=-\pi k
\qquad(k\ge1),
\qquad
B_0=B_1=0,
\]
so
\[
B_k=-\frac{\pi}{6}k(k^2-1).
\]
Consequently,
\[
D^2E_2(C)[c_kb,c_kb]
=
-\frac{4\pi^2}{3}k(k^2-1),
\]
and therefore
\[
\lambda_k^{(2)}
=
-\frac{4\pi}{3}|k|(k^2-1).
\]

Combining the two spectra, since
\[
E=E_1+E_2+4,
\]
\[
\lambda_k^E
=
\frac{2\pi}{3}|k|(k^2-1).
\]

\subsection{Full normal operators and kernel}

M\"obius invariance and stationarity transfer the binormal Hessians to the radial component, while reflection symmetry eliminates the mixed terms. More precisely,
\[
D^2E_i(C)[aC,aC]
=
D^2E_i(C)[ab,ab].
\]
Reflection gives
\[
D^2E_i(C)[aC,cb]=0.
\]

We now assemble the complete normal operators. We use the Fourier convention
\[
a_k=\frac1{2\pi}\int_0^{2\pi}a(s)e^{-iks}\,ds,
\qquad
c_k=\frac1{2\pi}\int_0^{2\pi}c(s)e^{-iks}\,ds.
\]
Thus
\[
a(s)=\sum_{k\in\mathbb Z}a_ke^{iks},
\qquad
c(s)=\sum_{k\in\mathbb Z}c_ke^{iks},
\]
and Parseval's identity reads
\[
\int_0^{2\pi}|a|^2\,ds=2\pi\sum_k|a_k|^2,
\qquad
\int_0^{2\pi}|c|^2\,ds=2\pi\sum_k|c_k|^2.
\]
Then, with the unnormalized \(L^2\)-inner product,
\[
D^2E_1(C)[u,u]
=
4\pi^2
\sum_{k\in\mathbb Z}
|k|(k^2-1)(|a_k|^2+|c_k|^2),
\]
and
\[
D^2E_2(C)[u,u]
=
-\frac{8\pi^2}{3}
\sum_{k\in\mathbb Z}
|k|(k^2-1)(|a_k|^2+|c_k|^2).
\]
Equivalently,
\[
H_1(C)=2\pi(|D_s|^3-|D_s|)I_2,
\]
\[
H_2(C)=-\frac{4\pi}{3}(|D_s|^3-|D_s|)I_2,
\]
and
\[
H_E(C)=\frac{2\pi}{3}(|D_s|^3-|D_s|)I_2.
\]
Thus
\[
H_1(C)=3H_E(C),\qquad H_2(C)=-2H_E(C).
\]

The normal kernel is
\[
\mathcal K
=
\left\{
aC+cb:\;
a,c\in\operatorname{span}_{\mathbb R}\{1,\cos s,\sin s\}
\right\}.
\]

This proves Theorem~\ref{thm:complete-hessians}.

\section{Vanishing Through Third Order}

Define
\[
\mathcal F=2E_1+3E_2+2\pi^2.
\]
Then
\[
\mathcal F(C)=0,
\qquad
D\mathcal F(C)=0.
\]
The Hessian relation gives
\[
D^2\mathcal F(C)=2H_1(C)+3H_2(C)=0
\]
on the normal space, and reparametrization invariance eliminates tangential slots.

For normal \(u,v,w\), M\"obius invariance under
\[
J(x)=\frac{(x_1,x_2,-x_3)}{|x|^2}
\]
gives
\[
D^3\mathcal F(C)[u,v,w]
=
D^3\mathcal F(C)[-u,-v,-w]
=
-
D^3\mathcal F(C)[u,v,w].
\]
Hence
\[
D^3\mathcal F(C)[u,v,w]=0.
\]
Every third derivative containing a tangential slot vanishes by reparametrization invariance.

Therefore
\[
D\mathcal F(C)=D^2\mathcal F(C)=D^3\mathcal F(C)=0.
\]

For a fixed smooth variation \(u\),
\[
\mathcal F(C+\varepsilon u)
=
\frac{\varepsilon^4}{24}
D^4\mathcal F(C)[u,u,u,u]
+
o(\varepsilon^4).
\]
For a binormal variation, the function is even in \(\varepsilon\), but we do not require a sixth-order remainder.

This proves Theorem~\ref{thm:lower-jets}.

\begin{lemma}[Quartic annihilation of symmetry directions]\label{lem:symmetry-slot}

Let \(X(f)\) be the infinitesimal vector field on the space of parametrized curves generated by a smooth one-parameter family of M\"obius transformations. M\"obius invariance gives
\[
D\mathcal F(f)[X(f)]=0.
\]
Differentiating this identity three times at \(f=C\) in arbitrary smooth directions \(u_1,u_2,u_3\), every term involving a derivative of \(X\) is multiplied by one of
\[
D\mathcal F(C),\qquad D^2\mathcal F(C),\qquad D^3\mathcal F(C),
\]
and hence vanishes. Therefore
\[
D^4\mathcal F(C)[u_1,u_2,u_3,X(C)]=0.
\]
The same conclusion holds when the last slot is tangential, by differentiating the reparametrization invariance identity. Since every element of \(\mathcal K\) is the normal part of an infinitesimal M\"obius direction, the difference being tangential, symmetry of the fourth derivative yields
\[
D^4\mathcal F(C)[u_1,u_2,u_3,k]=0
\qquad(k\in\mathcal K).
\]
Thus low normal modes may be removed not only from pure quartic terms but from every mixed quartic slot.

\end{lemma}

\section{Quartic Density and Fourier Resonances}

For a binormal variation \(u=cb\), let
\[
Q_4[cb]
:=
[\varepsilon^4]\mathcal F(C+\varepsilon cb)
=
\frac1{24}D^4\mathcal F(C)[cb,cb,cb,cb].
\]

With
\[
p=c'(s),\qquad q=c'(t),\qquad z=c(s)-c(t),
\]
expansion of the arbitrary-parameter formulas to fourth order gives the following expression; Appendix~D justifies differentiation under the integral sign and its diagonal extension.

\begin{proposition}[Quartic density]\label{prop:quartic-density}

\begin{equation}\label{eq:scalar-q4-integral}
Q_4[cb]
=
\iint K_4[c](s,t)\,ds\,dt,
\end{equation}
where
\begin{equation}\label{eq:scalar-K4}
\begin{aligned}
K_4[c]
={}&
-\frac{(p^2-q^2)^2}{4D}
-\frac{(p^2+q^2+10pq)z^2}{D^2}
\\
&+
\frac{12\sin h(p+q)z^3}{D^3}
-\frac{(7+5\cos h)z^4}{D^3}.
\end{aligned}
\end{equation}

The displayed terms must be combined before interpreting the diagonal behavior.

\end{proposition}

\subsection{Finite-part coefficients}

For the Fourier calculation we use the zero-mean periodic finite-part extensions of \(D^{-r}\). Their distributional normalization and the passage back to the ordinary combined integral are recorded in Appendix~A. The corresponding Fourier coefficients are
\[
J_1(k)=-|k|,
\]
\[
J_2(k)=\frac{1-k^2}{6}J_1(k)
=
\frac{|k|^3-|k|}{6},
\]
and
\[
J_3(k)=\frac{4-k^2}{20}J_2(k)
=
-\frac{|k|^5-5|k|^3+4|k|}{120}.
\]
These are only a bookkeeping device for the polarized combined density.

Define
\[
S_3(k)
=
\frac{J_3(k+1)-J_3(k-1)}2
=
-\frac{\operatorname{sgn}(k)k^2(k^2-1)}{24},
\]
and
\[
H_3(k)
=
7J_3(k)+\frac52(J_3(k+1)+J_3(k-1))
=
-\frac{|k|(|k|^2-1)(6|k|^2+1)}{60}.
\]

\subsection{Resonances}

Translation invariance implies frequency conservation
\[
k_1+k_2+k_3+k_4=0.
\]
For \(a,b,c\in\mathbb Z_{\ge1}\), all \(3+1\) resonances vanish:
\[
D^4\mathcal F(C)
[
e^{ias}b,
e^{ibs}b,
e^{ics}b,
e^{-i(a+b+c)s}b
]
=0.
\]

For \(a+b=c+d\) and
\[
1\le a\le c\le d\le b,
\]
the surviving scalar \(2+2\) symbol is
\begin{equation}\label{eq:four-wave-symbol}
\begin{aligned}
&D^4\mathcal F(C)
[
e^{ias}b,
e^{ibs}b,
e^{-ics}b,
e^{-ids}b
]
\\
&=
\frac{16\pi^2}{15}
a(a^2-1)
(2a^2+5ab+5b^2-5cd-3).
\end{aligned}
\end{equation}

For \(m\in\mathbb Z_{\ge1}\), let
\[
u_m=\cos(ms)b.
\]
\[
Q_4[u_m]
=
\frac{\pi^2}{60}
m(m^2-1)(7m^2-3).
\]

For \(1\le m<n\),
\[
D^4\mathcal F(C)
[u_m,u_m,u_n,u_n]
=
\frac{4\pi^2}{15}
m(m^2-1)(5n^2+2m^2-3).
\]

A genuine resonance is
\[
2+4=3+3,
\]
for which
\[
D^4\mathcal F(C)
[
e^{2is}b,
e^{4is}b,
e^{-3is}b,
e^{-3is}b
]
=
512\pi^2.
\]

\section{Scalar Fixed-Sum Positivity}

Let
\[
c(s)=c_0+\sum_{k\ge1}
(z_ke^{iks}+\overline{z_k}e^{-iks}).
\]
For fixed \(S\ge4\), let
\[
N=\left\lfloor\frac S2\right\rfloor,
\qquad
I_S:=\{2,\ldots,N\},
\]and, for \(i\in I_S\),
\[
w_i^{(S)}=z_iz_{S-i}.
\]
Set
\[
\mu_i^{(S)}
=
\begin{cases}
2,&2i<S,\\
1,&2i=S.
\end{cases}
\]
and let \(D_\mu=\operatorname{diag}(\mu_i^{(S)})_{i\in I_S}\). Then
\begin{equation}\label{eq:scalar-fixed-sum-Q}
Q_4[cb]
=
\frac{4\pi^2}{15}
\sum_{S\ge4}
(D_\mu w^{(S)})^*
M^{(S)}
(D_\mu w^{(S)}),
\end{equation}
where
\begin{equation}\label{eq:scalar-fixed-sum-M}
M_{ij}^{(S)}
=
r(r^2-1)
[5S^2-5S(r+s)+2r^2+5s^2-3],
\end{equation}
with
\[
r=\min(i,j),\qquad s=\max(i,j).
\]

For
\[
2\le n\le N,
\]
let
\[
M_n=(M_{ij})_{2\le i,j\le n},
\]
so that \(M_n\) is an \((n-1)\times(n-1)\) leading principal block. Define
\[
q_2=e_2,\qquad q_3=e_3,
\]
and for \(j\ge4\),
\[
q_j
=
(S-2j+1)e_{j-2}
-
2(S-2j+2)e_{j-1}
+
(S-2j+3)e_j.
\]
Let
\[
Q_n=(q_2,\ldots,q_n),\qquad
T_n=Q_n^TM_nQ_n.
\]
A direct computation shows that \(T_n\) is tridiagonal. Appendix~B records its entries, treats the exceptional coefficient \(B_4\), and proves the continuant formula. The resulting determinant is
\begin{equation}\label{eq:scalar-det}
\det M_n
=
5\,15^{n-2}n!(n-1)!
\prod_{j=2}^{n-1}(S-j)^2
\bigl[3(S-n)^2+n^2-1\bigr].
\end{equation}
Since \(2\le n\le\lfloor S/2\rfloor\), every factor in \eqref{eq:scalar-det} is positive. Hence all leading principal minors are positive, and Sylvester's criterion gives
\[
M^{(S)}>0
\]
for every \(S\ge4\).

Therefore
\[
Q_4[cb]\ge0,
\]
with equality iff
\[
c\in\operatorname{span}_{\mathbb R}\{1,\cos s,\sin s\}.
\]

\section{Full Normal Quartic Nondegeneracy}

We now treat
\[
u=aC+cb.
\]

\subsection{M\"obius reduction to a two-component binormal problem}

Embed
\[
\mathbb R^3=\{x_4=0\}\subset\mathbb R^4
\]
and regard
\[
C(s)=(\cos s,\sin s,0,0).
\]

The normal bundle is spanned by
\[
C,\qquad e_3,\qquad e_4.
\]

Reflections in \(x_3\) and \(x_4\) show that the \(e_3\)- and \(e_4\)-first variations vanish. Define
\[
T_4(x)
=
e_4+\frac{2(x-e_4)}{|x-e_4|^2}.
\]
Then
\[
T_4(C)=C,
\qquad
DT_{4,C}[C]=e_4.
\]
Hence the radial first variation vanishes as well.

The fixed-binormal Hessian calculation is dimension independent, so the \(e_3\)- and \(e_4\)-directions have the same multipliers:
\[
H_1^{e_3}=H_1^{e_4}
=
2\pi(|D_s|^3-|D_s|),
\]
\[
H_2^{e_3}=H_2^{e_4}
=
-\frac{4\pi}{3}(|D_s|^3-|D_s|).
\]
Coordinate reflections eliminate mixed \(e_3/e_4\) terms. Stationarity and the second-order chain rule for \(T_4\) give the same cancellation in the radial direction. Therefore
\[
D^2\mathcal F_{\mathbb R^4}(C)=0.
\]

Now define
\[
J_4(x)
=
\frac{(x_1,x_2,-x_3,-x_4)}{|x|^2}.
\]
It fixes \(C\) and satisfies
\[
DJ_{4,C}[\tau]=\tau,
\]
\[
DJ_{4,C}[C]=-C,
\qquad
DJ_{4,C}[e_3]=-e_3,
\qquad
DJ_{4,C}[e_4]=-e_4.
\]
Hence
\[
DJ_{4,C}|_{N_C}=-I.
\]
Because the first two derivatives of \(\mathcal F\) vanish, M\"obius invariance implies
\[
D^3\mathcal F_{\mathbb R^4}(C)=0
\]
on purely normal slots, while reparametrization invariance removes tangential slots.

Thus
\[
D\mathcal F_{\mathbb R^4}(C)
=
D^2\mathcal F_{\mathbb R^4}(C)
=
D^3\mathcal F_{\mathbb R^4}(C)
=
0.
\]

With the lower-order jets established in \(\mathbb R^4\), set
\[
\gamma(\varepsilon)=T_4(C+\varepsilon u),
\qquad
v_j=\gamma^{(j)}(0).
\]
The fourth-order chain rule gives
\[
\begin{aligned}
\frac{d^4}{d\varepsilon^4}\mathcal F(\gamma(\varepsilon))\bigg|_{\varepsilon=0}
={}&D^4\mathcal F(C)[v_1,v_1,v_1,v_1]\\
&+6D^3\mathcal F(C)[v_1,v_1,v_2]\\
&+3D^2\mathcal F(C)[v_2,v_2]\\
&+4D^2\mathcal F(C)[v_1,v_3]\\
&+D\mathcal F(C)[v_4].
\end{aligned}
\]
All terms except the first vanish by the full-space lower-jet identities. For
\[
u=aC+cb,
\]
we have
\[
v_1=DT_{4,C}[u]=ae_4+cb.
\]
Since \(\mathcal F\) is M\"obius invariant and its Euclidean formula is unchanged when a curve in \(\mathbb R^3\) is regarded as a curve in \(\mathbb R^4\), we obtain
\[
D^4\mathcal F_{\mathbb R^3}(C)
[(aC+cb)^4]
=
D^4\mathcal F_{\mathbb R^4}(C)
[(ae_4+cb)^4].
\]

Thus the full normal problem in \(\mathbb R^3\) reduces to a two-component binormal problem.

\subsection{Vector-valued quartic form and fixed-sum decomposition}

Let
\[
v=(a,c):\mathbb T\to\mathbb R^2,
\]
and identify \(v=(a,c)\) with the fixed-binormal field \(ae_4+ce_3\) in \(\mathbb R^4\). Set
\[
p=v'(s),\qquad q=v'(t),\qquad z=v(s)-v(t).
\]
Then
\begin{equation}\label{eq:vector-K4}
\begin{aligned}
K_4^{\rm vec}[v]
={}&
-\frac{(|p|^2-|q|^2)^2}{4D}
\\
&-
\frac{
(|p|^2+|q|^2+4p\cdot q)|z|^2
+
6(p\cdot z)(q\cdot z)
}{D^2}
\\
&+
\frac{
12\sin h((p+q)\cdot z)|z|^2
}{D^3}
\\
&-
\frac{(7+5\cos h)|z|^4}{D^3}.
\end{aligned}
\end{equation}

Write
\[
v(s)=v_0+\sum_{k\ge1}
(z_ke^{iks}+\overline{z_k}e^{-iks}),
\qquad z_k\in\mathbb C^2.
\]

For fixed \(S\ge4\), let
\[
I_S=\{2,\ldots,\lfloor S/2\rfloor\}.
\]
For \(p\in I_S\), set
\begin{equation}\label{eq:Wp-def}
\mu_p=
\begin{cases}
2,&2p<S,\\
1,&2p=S,
\end{cases}
\qquad
W_p=\mu_pz_pz_{S-p}^T.
\end{equation}
Decompose
\[
W_p
=
\frac{\operatorname{tr}W_p}{2}I_2
+
W_p^0
+
W_p^-,
\]
where
\[
W_p^-=
\frac{W_p-W_p^T}{2},
\qquad
W_p^0=
\frac{W_p+W_p^T}{2}
-
\frac{\operatorname{tr}W_p}{2}I_2.
\]
Here \(T\) denotes the ordinary transpose. Set
\[
t_p=\frac{\operatorname{tr}W_p}{\sqrt2}.
\]
The complex Frobenius inner product used below is
\[
\langle A,B\rangle_F
:=
\operatorname{tr}(AB^*)
=
\sum_{\alpha,\beta}A_{\alpha\beta}\overline{B_{\alpha\beta}}.
\]

Then
\begin{equation}\label{eq:RUV-decomp}
\begin{aligned}
Q_4[v]
=
\frac{2\pi^2}{15}
\sum_{S\ge4}\sum_{p,q\in I_S}
\Big[
&
R_{pq}t_p\overline{t_q}
\\
&+
U_{pq}\langle W_p^0,W_q^0\rangle_F
\\
&+
V_{pq}\langle W_p^-,W_q^-\rangle_F
\Big].
\end{aligned}
\end{equation}

Let
\[
r=\min(p,q),\qquad s=\max(p,q),\qquad
g(r)=r(r^2-1).
\]
Then
\[
U_{pq}
=
g(r)
[10S^2-15Sr+6r^2-4],
\]
\[
V_{pq}
=
5g(r)(2S-r)(S-2s),
\]
and
\[
R_{pq}
=
g(r)
[10S^2-5Sr-20Ss+2r^2+20s^2-8].
\]

\subsection{Positivity of the fixed-sum blocks and the equality case}

\medskip\noindent\textbf{The \(U\)-block.}

Since
\[
U_{pq}=f(\min(p,q)),
\]
with
\[
f(r)=g(r)[10S^2-15Sr+6r^2-4],
\]
and
\[
f(r)-f(r-1)
=
30r(r-1)(S-r)(S-r+1)>0,
\]
the cumulative-sum factorization gives
\[
U>0.
\]

\medskip\noindent\textbf{The \(V\)-block.}

Write
\[
V_{pq}
=
p_0(\min(p,q))q_0(\max(p,q)),
\]
where
\[
p_0(r)=5g(r)(2S-r),
\qquad
q_0(r)=S-2r.
\]
On noncentral indices,
\[
p_0>0,\qquad q_0>0,
\]
and
\[
p_0(r)-p_0(r-1)
=
10r(r-1)(3S-2r+1)>0.
\]
Hence \(p_0/q_0\) is increasing and \(V\) is positive definite on the noncentral antisymmetric coordinates that actually occur.

If \(S\) is even, the formally included central row and column vanish, so the full matrix with that coordinate is not positive definite. This creates no additional zero direction in the quartic form because
\[
W_{S/2}^-=0
\]
identically.

\medskip\noindent\textbf{The \(R\)-block.}

Let
\[
N=\left\lfloor\frac S2\right\rfloor.
\]
For \(2\le j\le N\), define
\[
\widetilde q_j
=
(S-2j+1)e_{j-2}
-
2(S-2j+2)e_{j-1}
+
(S-2j+3)e_j,
\]
with \(e_0=e_1=0\), and set
\[
\widetilde Q=(\widetilde q_2,\ldots,\widetilde q_N),
\qquad
\widetilde T=\widetilde Q^TR\widetilde Q.
\]
Then \(\widetilde T\) is tridiagonal. For \(3\le j\le N\), its off-diagonal entries are
\[
\widetilde B_j
=
-30(j-2)(j-1)
(S-j+2)(S-j+1)
(S-2j+5)(S-2j+1)<0.
\]
For \(2\le j<N\), the row sums are
\[
\sigma_j
=
120(j-1)(S-j+1)
[(S-j+1)^2+(j-1)^2-1]>0.
\]
For the terminal row,
\[
\sigma_N
=
30(N-1)(N+1)(5N^2+4)
\]
if \(S=2N\), and
\[
\sigma_N
=
240(N-1)(N+2)(N^2+N+2)
\]
if \(S=2N+1\).

Therefore
\begin{equation}\label{eq:R-block-positive}
x^*\widetilde T x
=
\sum_{j=3}^{N}
(-\widetilde B_j)|x_j-x_{j-1}|^2
+
\sum_{j=2}^{N}
\sigma_j|x_j|^2>0
\end{equation}
for every nonzero \(x\). Moreover, \(\widetilde Q\) is triangular with
\[
\det\widetilde Q
=
\prod_{j=2}^{N}(S-2j+3)>0,
\]
so the congruence is invertible. Hence
\[
R>0.
\]

\medskip\noindent\textbf{Equality case.}

If \(z_m\ne0\) for some \(m\ge2\), consider the block \(S=2m\). Its central matrix is
\[
W_m=z_mz_m^T\ne0.
\]
Since it is symmetric,
\[
W_m^-=0,
\]
but its trace and symmetric-traceless parts cannot both vanish. Since \(R>0\) and \(U>0\), that block contributes strictly positively.

For trigonometric polynomials it follows that
\[
Q_4[v]=0
\iff
z_m=0\quad(m\ge2).
\]
For a smooth field, the fixed-sum expansion is absolutely summable as explained in Appendix D. If some fixed coefficient \(z_m\ne0\) with \(m\ge2\), then the block \(S=2m\) is unchanged in all sufficiently large Fourier truncations and contributes a strictly positive amount. Hence the same equality characterization holds for arbitrary smooth fields.

Returning to \(\mathbb R^3\),
\[
D^4\mathcal F(C)[u,u,u,u]\ge0,
\]
with equality iff
\[
u\in\mathcal K.
\]

This proves Theorem~\ref{thm:quartic}.

\section{Discussion and Open Problems}

The functional
\[
\mathcal F
=
3(E-4)-(E_1-2\pi^2)
\]
measures the first deviation of \(E_1-E_1(C)\) from \(3(E-E(C))\). These quantities agree through third order at the circle.

For every fixed normal direction \(u\notin\mathcal K\),
\[
\mathcal F(C+\varepsilon u)
=
\frac{\varepsilon^4}{24}
D^4\mathcal F(C)[u^4]
+
o(\varepsilon^4),
\]
so
\[
\mathcal F(C+\varepsilon u)>0
\]
for sufficiently small nonzero \(\varepsilon\).

This is not a uniform coercivity statement.

For
\[
c_m=m^{-3/2}\cos(ms),
\]
one has
\[
\|c_m\|_{H^{3/2}}\asymp1,
\]
but
\[
Q_4[c_mb]\sim \frac{7\pi^2}{60m}\to0.
\]
Therefore no estimate of the form
\[
Q_4[cb]\ge c\|c\|_{H^{3/2}}^4
\]
can hold uniformly modulo the low modes.

The pure-mode scaling
\[
Q_4[\cos(ms)b]\asymp m^5
\]
suggests an \(H^{5/4}\)-type scale, but no multimode coercivity estimate is proved.

Likewise, quartic directional positivity does not imply a nonlinear local minimum in an infinite-dimensional neighborhood, and we do not claim a \L{}ojasiewicz--Simon inequality or gradient-flow stability for \(\mathcal F\).

\appendix

\section{Finite-Part Fourier Identities and the Four-Wave Tensor}

Let
\[
D(h)=2-2\cos h.
\]

This appendix fixes the distributional convention used for the singular kernels and explains why the resulting finite-part calculation agrees with the ordinary integral of the combined quartic density.

\subsection{Periodic finite-part distributions and the ordinary integral}

Let
\[
F_1
:=
-\frac12\partial_h\operatorname{pv}\cot\frac h2.
\]
Define recursively
\[
F_2:=\frac{F_1''+F_1}{6},
\qquad
F_3:=\frac{F_2''+4F_2}{20}.
\]
These are periodic distributions. We normalize their Fourier coefficients by
\[
\langle F_r,e^{ikh}\rangle=\pi J_r(k).
\]
Equivalently, the usual \(1/(2\pi)\)-normalized Fourier coefficient of \(F_r\) is \(J_r(k)/2\).

With this convention,
\[
J_1(k)=-|k|,
\]
\[
J_2(k)=\frac{|k|^3-|k|}{6},
\]
and
\[
J_3(k)
=
-\frac{|k|^5-5|k|^3+4|k|}{120}.
\]
In particular,
\[
J_r(0)=0.
\]

The recurrence is the Fourier form of
\[
F_2=\frac{F_1''+F_1}{6},
\qquad
F_3=\frac{F_2''+4F_2}{20}.
\]

A key feature of this normalization is the multiplication identity
\[
D\,F_1=1,\qquad
D\,F_2=F_1,\qquad
D\,F_3=F_2,
\]
where \(D\) denotes multiplication by the smooth function \(D(h)=2-2\cos h\). Hence
\[
D^3F_3=1.
\]

Define also
\[
S_3(k)
=
\frac{J_3(k+1)-J_3(k-1)}2
=
-\frac{\operatorname{sgn}(k)k^2(k^2-1)}{24},
\]
and
\[
H_3(k)
=
7J_3(k)
+\frac52[J_3(k+1)+J_3(k-1)]
=
-\frac{|k|(|k|^2-1)(6|k|^2+1)}{60}.
\]

\medskip\noindent\textbf{Relation to the ordinary combined integral.} The distributions \(F_r\) are not used to assign separate geometric meanings to the individually divergent terms of the physical-space formula. They are only a common algebraic representation after the quartic density has been combined.

For trigonometric-polynomial inputs, put the polarized combined quartic density over the common denominator \(D^3\):
\[
K(h)=\frac{N(h)}{D(h)^3}.
\]
The numerator \(N\) is a trigonometric polynomial. By the desingularization proved in Appendix D, \(K\) extends continuously across \(h=0\). Since \(D^3\) has a zero of order six, \(N\) has a zero of order at least six at \(h=0\). Equivalently, with \(x=e^{ih}\),
\[
D^3=-\frac{(x-1)^6}{x^3},
\]
and the Laurent polynomial \(N(x)\) is divisible by \((x-1)^6\). Thus there is a trigonometric polynomial \(\widetilde K\) such that
\[
N=D^3\widetilde K,
\]
and \(\widetilde K=K\) off the diagonal, hence everywhere by continuity.

Therefore
\[
\langle F_3,N\rangle
=
\langle F_3,D^3\widetilde K\rangle
=
\langle D^3F_3,\widetilde K\rangle
=
\int_{-\pi}^{\pi}\widetilde K(h)\,dh.
\]
Thus the finite-part Fourier bookkeeping gives exactly the ordinary integral of the combined quartic density.

This is the only sense in which \(F_1,F_2,F_3\) are used below.

\subsection{Four-wave tensor and resonance chambers}

For
\[
k_1+k_2+k_3+k_4=0,
\]
write
\[
k_A=\sum_{j\in A}k_j.
\]
Define
\[
\begin{aligned}
\Phi(k_1,k_2,k_3,k_4)
={}&
\frac12k_1k_2k_3k_4J_1(k_1+k_2)
\\
&+
k_1k_2
\sum_{A\subset\{3,4\}}
(-1)^{2-|A|}
\left[
J_2(k_1+k_2+k_A)+J_2(k_A)
\right.
\\
&\hspace{5.7cm}\left.
+4J_2(k_1+k_A)
\right]
\\
&+
6k_1k_3
\sum_{A\subset\{2,4\}}
(-1)^{2-|A|}
J_2(k_1+k_A)
\\
&+
12k_1
\sum_{A\subset\{2,3,4\}}
(-1)^{3-|A|}
\left[
S_3(k_1+k_A)+S_3(k_A)
\right]
\\
&-
\sum_{A\subset\{1,2,3,4\}}
(-1)^{4-|A|}
H_3(k_A).
\end{aligned}
\]

Throughout Appendices A and C, the constant vectors \(\xi_j\) belong to the complexification of the fixed binormal plane
\[
\mathcal B=\operatorname{span}_{\mathbb R}\{e_3,e_4\}\subset\mathbb R^4.
\]
Then, with the Euclidean product extended complex bilinearly at the polarization stage,
\[
\begin{aligned}
&
\frac{
D^4\mathcal F(C)
[
e^{ik_1s}\xi_1,\ldots,e^{ik_4s}\xi_4
]
}{
2\pi^2
}
\\
&=
\sum_{\sigma\in\mathfrak S_4}
\Phi(k_{\sigma(1)},\ldots,k_{\sigma(4)})
(\xi_{\sigma(1)}\cdot\xi_{\sigma(2)})
(\xi_{\sigma(3)}\cdot\xi_{\sigma(4)}).
\end{aligned}
\]

If the total frequency is nonzero, the tensor vanishes.

\medskip\noindent\textbf{The \(3+1\) cancellation.}

For positive integers \(a,b,c\), set
\[
(k_1,k_2,k_3,k_4)
=
(a,b,c,-a-b-c).
\]
After fixing the signs of all subset sums, each of the three invariant contraction coefficients becomes an ordinary polynomial. Direct simplification gives
\[
T_{12|34}=T_{13|24}=T_{14|23}=0.
\]
Hence the \(3+1\) resonance vanishes as a vector-valued tensor, not merely after scalar contraction:
\[
D^4\mathcal F(C)
[
e^{ias}\xi_1,
e^{ibs}\xi_2,
e^{ics}\xi_3,
e^{-i(a+b+c)s}\xi_4
]
=0.
\]

\medskip\noindent\textbf{The \(2+2\) chamber.}

For
\[
a+b=c+d,
\qquad
1\le a\le c\le d\le b,
\]
write
\[
a=A,\qquad
c=A+X,\qquad
d=A+X+Y,\qquad
b=A+2X+Y,
\]
with
\[
A\ge1,\qquad X,Y\ge0.
\]
All absolute values in the Fourier coefficients then have fixed signs, and the tensor reduces to an ordinary polynomial identity.

In the scalar binormal case,
\[
\begin{aligned}
&D^4\mathcal F(C)
[
e^{ias}b,
e^{ibs}b,
e^{-ics}b,
e^{-ids}b
]
\\
&=
\frac{16\pi^2}{15}
a(a^2-1)
(2a^2+5ab+5b^2-5cd-3).
\end{aligned}
\]
The boundary cases \(X=0\), \(Y=0\), and \(X=Y=0\) are included in the same polynomial identity.

The real-mode and fixed-sum formulas in Sections 5 and 6 follow by the stated \(1/2\)-Fourier coefficients and the definition
\[
Q_4=\frac1{24}D^4\mathcal F(C).
\]

\section{Scalar Determinant Calculation}

For fixed \(S\), define
\[
\begin{aligned}
M_{ij}
&=
r(r^2-1)
[5S^2-5S(r+s)+2r^2+5s^2-3],\\
r&=\min(i,j),\qquad s=\max(i,j).
\end{aligned}
\]
For \(2\le n\le\lfloor S/2\rfloor\), let
\[
M_n=(M_{ij})_{2\le i,j\le n}.
\]
The weighted difference vectors
\[
q_2=e_2,\qquad q_3=e_3,
\]
\[
q_j
=
(S-2j+1)e_{j-2}
-
2(S-2j+2)e_{j-1}
+
(S-2j+3)e_j,
\qquad 4\le j\le n,
\]
annihilate the quadratic dependence on the larger index. With
\[
Q_n=(q_2,\ldots,q_n),
\qquad
T_n=Q_n^TM_nQ_n,
\]
the matrix \(T_n\) is tridiagonal.

Its diagonal entries are
\[
A_2=30(S^2-4S+5),
\qquad
A_3=120(S^2-6S+12),
\]
and, for \(j\ge4\),
\[
A_j
=
30(j-1)^2(S-j+1)^2\bigl[(S-2j+2)^2+1\bigr].
\]
The off-diagonal entries are
\[
B_3=30(S^2-5S+10),
\qquad
B_4=-90(S-7)(S-3)(S-2),
\]
and, for \(j\ge5\),
\[
B_j
=
-15(j-2)(j-1)(S-j+2)(S-j+1)
(S-2j+5)(S-2j+1).
\]
The coefficient \(B_4\) is exceptional and is not obtained by substituting \(j=4\) into the last formula.

Let
\[
\Theta_1=1,\qquad
\Theta_2=A_2,\qquad
\Theta_n=\det T_n\quad(n\ge2).
\]
Then
\begin{equation}\label{eq:continuant}
\Theta_n=A_n\Theta_{n-1}-B_n^2\Theta_{n-2},
\qquad n\ge3.
\end{equation}
The initial cases are
\[
\Theta_2=30(S^2-4S+5),
\]
\[
\Theta_3
=
900(S-2)^2\bigl(3(S-3)^2+8\bigr),
\]
and
\[
\Theta_4
=
162000(S-2)^2(S-3)^2(S-5)^2
\bigl(3(S-4)^2+15\bigr),
\]
where the last identity uses the exceptional coefficient \(B_4\).

For completeness, write
\[
\Theta_n=P_nR_n,
\]
where
\[
P_n
=
5\,15^{n-2}n!(n-1)!
\prod_{j=2}^{n-1}(S-j)^2
\prod_{j=4}^{n}(S-2j+3)^2,
\qquad
R_n=3(S-n)^2+n^2-1.
\]
For \(n\ge4\),
\[
\frac{P_n}{P_{n-1}}
=
15n(n-1)(S-n+1)^2(S-2n+3)^2.
\]
For \(n\ge5\), substitution into \eqref{eq:continuant} reduces the induction to
\begin{equation}\label{eq:continuant-reduced}
\begin{aligned}
n(S-2n+3)^2R_n
={}&
2(n-1)\bigl((S-2n+2)^2+1\bigr)R_{n-1}
\\
&-
(n-2)(S-2n+1)^2R_{n-2},
\end{aligned}
\end{equation}
which follows by expanding \(R_j=3(S-j)^2+j^2-1\).
This completes the induction for \(\Theta_n=P_nR_n\).
Finally,
\[
\det Q_n
=
\prod_{j=4}^{n}(S-2j+3),
\]
so the congruence \(T_n=Q_n^TM_nQ_n\) gives
\[
\det M_n
=
5\,15^{n-2}n!(n-1)!
\prod_{j=2}^{n-1}(S-j)^2
R_n,
\]
which is \eqref{eq:scalar-det}.

\section{Vector Fourier Tensor and the \texorpdfstring{\(R/U/V\)}{R/U/V} Decomposition}

As in Appendix A, the vectors \(\xi_j\) belong to the complexification of the fixed binormal plane
\[
\mathcal B=\operatorname{span}_{\mathbb R}\{e_3,e_4\}\subset\mathbb R^4.
\]
For fixed total frequency \(S\ge4\), let
\[
I_S=\{2,\ldots,\lfloor S/2\rfloor\},
\]
and let \(p,q\in I_S\). Consider the frequencies
\[
p,\quad S-p,\quad -q,\quad -(S-q).
\]
The tensor has the form
\[
\begin{aligned}
&D^4\mathcal F(C)
[
e^{ips}\xi_1,
e^{i(S-p)s}\xi_2,
e^{-iqs}\xi_3,
e^{-i(S-q)s}\xi_4
]
\\&=
\frac{8\pi^2}{15}
\Big[
\alpha_{pq}
(\xi_1\cdot\xi_2)(\xi_3\cdot\xi_4)
+
\beta_{pq}
(\xi_1\cdot\xi_3)(\xi_2\cdot\xi_4)
+
\gamma_{pq}
(\xi_1\cdot\xi_4)(\xi_2\cdot\xi_3)
\Big].
\end{aligned}
\]

With
\[
r=\min(p,q),\qquad s=\max(p,q),\qquad g(r)=r(r^2-1),
\]
\[
\alpha_{pq}
=
-g(r)[-5Sr+10Ss+2r^2-10s^2+2].
\]
Set
\[
U_{pq}=\beta_{pq}+\gamma_{pq},
\qquad
V_{pq}=\beta_{pq}-\gamma_{pq}.
\]
Then
\[
U_{pq}
=
g(r)[10S^2-15Sr+6r^2-4],
\]
\[
V_{pq}
=
5g(r)(2S-r)(S-2s).
\]
The scalar consistency identity is
\[
\alpha_{pq}+\beta_{pq}+\gamma_{pq}
=
2M_{pq}.
\]

For
\[
W_p=\mu_pz_pz_{S-p}^T,
\]
where \(T\) denotes the ordinary transpose, decompose
\[
W_p
=
\frac{\operatorname{tr}W_p}{2}I_2
+
W_p^0+W_p^-,
\]
with
\[
W_p^-=
\frac{W_p-W_p^T}{2},
\qquad
W_p^0=
\frac{W_p+W_p^T}{2}
-
\frac{\operatorname{tr}W_p}{2}I_2.
\]
For the resulting real quartic form we use the Hermitian Frobenius pairing
\[
\langle A,B\rangle_F
=
\operatorname{tr}(AB^*)
=
\sum_{\alpha,\beta}A_{\alpha\beta}\overline{B_{\alpha\beta}}.
\]
Then the trace coefficient is
\[
R_{pq}=2\alpha_{pq}+U_{pq},
\]
that is,
\[
R_{pq}
=
g(r)
[10S^2-5Sr-20Ss+2r^2+20s^2-8].
\]

\section{Local \texorpdfstring{\(C^4\)}{C4}-Regularity and Diagonal Desingularization}

This appendix justifies differentiation under the double integral and the passage from trigonometric polynomials to smooth variations. No optimal regularity statement is intended.

\subsection{Desingularization and local \texorpdfstring{\(C^4\)}{C4}-regularity}

Write
\[
t=s-h.
\]
For \(h\ne0\), define
\[
m_f(s,h)
=
\frac{f(s)-f(s-h)}h
=
\int_0^1 f'(s-\theta h)\,d\theta.
\]
Define
\[
A_f(s,h)
=
\frac{f'(s)-m_f(s,h)}h
=
\int_0^1(1-\theta)f''(s-\theta h)\,d\theta,
\]
\[
B_f(s,h)
=
\frac{f'(s-h)-m_f(s,h)}h
=
-\int_0^1\theta f''(s-\theta h)\,d\theta.
\]
Finally, with
\[
\tau_f=\frac{f'}{|f'|},
\]
set
\[
L_f(s,h)
=
\frac{\tau_f(s)-\tau_f(s-h)}h
=
\int_0^1\tau_f'(s-\theta h)\,d\theta.
\]

The integral formulas define \(m_f,A_f,B_f,L_f\) continuously at \(h=0\). If \(f\in C^2\), then these quantities are continuous in \((s,h)\).

For \(C\) and all curves in a sufficiently small \(C^1\)-neighborhood of \(C\),
\[
|f'|\ge c>0,
\qquad
|m_f|\ge c>0.
\]
The second inequality is the local chord estimate in desingularized form. Indeed, for \(0<|h|\le\pi\),
\[
|m_C(s,h)|
=
\frac{2|\sin(h/2)|}{|h|}
\ge
\frac2\pi,
\qquad
\|m_f-m_C\|_\infty
\le
\|f'-C'\|_\infty,
\]
so the bound persists in a sufficiently small \(C^1\)-neighborhood of \(C\).

\medskip\noindent\textbf{Desingularized energy densities.} Let \(\mathscr I_1(f;s,h)\) and \(\mathscr I_2(f;s,h)\) denote the arbitrary-parameter densities of \(E_1\) and \(E_2\).

Since
\[
f(s)-f(s-h)=h\,m_f,
\]
and
\[
1-\tau_f(s)\cdot\tau_f(s-h)
=
\frac12|\tau_f(s)-\tau_f(s-h)|^2,
\]
the first density becomes
\begin{equation}\label{eq:desing-I1}
\mathscr I_1(f;s,h)
=
\frac{|f'(s)||f'(s-h)|}{2|m_f(s,h)|^2}
|L_f(s,h)|^2.
\end{equation}

For the second density, let
\[
P_m^\perp
=
I-\frac{m\otimes m}{|m|^2}.
\]
Because
\[
P_{m_f}^\perp A_f
=
\frac1hP_{m_f}^\perp f'(s),
\qquad
P_{m_f}^\perp B_f
=
\frac1hP_{m_f}^\perp f'(s-h),
\]
we obtain
\begin{equation}\label{eq:desing-I2}
\mathscr I_2(f;s,h)
=
\frac{2}{|m_f(s,h)|^2}
\left\langle
P_{m_f(s,h)}^\perp A_f(s,h),
P_{m_f(s,h)}^\perp B_f(s,h)
\right\rangle.
\end{equation}

These formulas contain no negative power of \(h\). The apparent diagonal singularities of the original arbitrary-parameter formulas have been completely removed before any variation is taken.

\medskip\noindent\textbf{Local \(C^4\)-regularity.}

\begin{lemma}\label{lem:C4}

There exists a sufficiently small \(C^2\)-neighborhood \(\mathcal U\) of \(C\) such that \(E_1\) and \(E_2\) are \(C^4\) functionals on \(\mathcal U\). More precisely, for \(i=1,2\) and \(0\le j\le4\),
\[
\left|
D^j\mathscr I_i(f;s,h)
[u_1,\ldots,u_j]
\right|
\le
C_j
\prod_{\ell=1}^j\|u_\ell\|_{C^2},
\]
uniformly for
\[
f\in\mathcal U,\qquad
(s,h)\in\mathbb T\times[-\pi,\pi].
\]

\end{lemma}

\begin{proof}

The maps
\[
f\mapsto m_f,\quad
f\mapsto A_f,\quad
f\mapsto B_f
\]
are bounded linear maps from \(C^2\) to \(C^0(\mathbb T\times[-\pi,\pi])\).

On a \(C^1\)-neighborhood in which \(|f'|\ge c\), the map
\[
f\mapsto\tau_f=\frac{f'}{|f'|}
\]
is smooth as a map from \(C^2\) to \(C^1\), and therefore
\[
f\mapsto L_f
\]
is smooth from \(C^2\) to \(C^0\).

Likewise, on the set \(|m_f|\ge c\), the maps
\[
m\mapsto |m|^{-2},
\qquad
m\mapsto P_m^\perp
\]
are smooth with all derivatives uniformly bounded on bounded subsets separated from \(m=0\).

The desingularized formulas above are therefore finite compositions and products of smooth maps on bounded \(C^2\)-sets. Their Fréchet derivatives through order four satisfy the displayed uniform multilinear bound.

Integrating over the compact parameter set \(\mathbb T\times[-\pi,\pi]\) gives \(C^4\)-regularity of \(E_1,E_2\) and justifies differentiation under the integral sign through fourth order.

\end{proof}

\begin{corollary}\label{cor:chainrules}

All M\"obius and reparametrization chain-rule arguments used in Sections 2, 4, and 7 are legitimate for smooth finite-dimensional families near the circle, and the fourth derivative may be computed by differentiating the desingularized densities under the integral.

\end{corollary}

\subsection{Diagonal behavior and Fourier extension}

For a vector-valued binormal variation, set
\[
p=v'(s),\qquad A=v''(s).
\]
The combined quartic density extends continuously to the diagonal and satisfies
\[
K_4^{\rm vec}[v](s,s)
=
\frac12
\left[
|v'(s)|^2|v''(s)|^2
+
(v'(s)\cdot v''(s))^2
-
|v'(s)|^4
\right].
\]

The cleanest verification uses the desingularized component densities. On the diagonal,
\[
\mathscr I_1(f)(s,s)=\frac12|\tau_f'(s)|^2,
\qquad
\mathscr I_2(f)(s,s)=-\frac12|\tau_f'(s)|^2.
\]
Hence
\[
2\mathscr I_1(f)(s,s)+3\mathscr I_2(f)(s,s)
=-\frac12|\tau_f'(s)|^2.
\]
For
\[
f_\varepsilon=C+\varepsilon v
\]
one has
\[
f_\varepsilon'=\tau+\varepsilon v',
\qquad
f_\varepsilon''=-C+\varepsilon v'',
\]
and expanding the regular expression \(-\frac12|\tau_{f_\varepsilon}'|^2\) gives the displayed fourth-order diagonal coefficient.

As a secondary consistency check, one may also use
\[
q=p-hA+o(h),
\qquad
z=hp-\frac{h^2}{2}A+o(h^2),
\]
together with
\[
D(h)=h^2+O(h^4),\qquad
\sin h=h+O(h^3),\qquad
\cos h=1-\frac{h^2}{2}+O(h^4).
\]
The desingularized representation above supplies the uniform remainder control needed to justify this check despite the apparent singular powers in the unreduced formula.

In the scalar case,
\[
K_4[c](s,s)
=
(c'(s))^2(c''(s))^2
-\frac12(c'(s))^4.
\]

\medskip\noindent\textbf{Continuity of \(Q_4\) and Fourier truncation.}

The fourth derivative furnished by Lemma~\ref{lem:C4} satisfies
\[
|Q_4[v]|
\le
C\|v\|_{C^2}^4
\]
on bounded \(C^2\)-sets.

If \(v^{(N)}\) is the Fourier truncation of a smooth field \(v\), then
\[
v^{(N)}\to v
\qquad\text{in }C^2,
\]
and therefore
\[
Q_4[v^{(N)}]\to Q_4[v].
\]

Thus every Fourier identity proved first for trigonometric polynomials extends to arbitrary smooth fields.

Alternatively, for smooth fields the Fourier coefficients decrease faster than any power while the fixed-sum matrix coefficients grow only polynomially (of order at most five), so the corresponding quartic Fourier series is absolutely summable.

\medskip\noindent\textbf{A convenient nonoptimal Sobolev extension.}

In one dimension,
\[
H^3(\mathbb T)\hookrightarrow C^2(\mathbb T).
\]
Therefore
\[
|Q_4[v]|
\le
C\|v\|_{C^2}^4
\le
C'\|v\|_{H^3}^4.
\]
Hence the quartic form extends continuously to \(H^3\).

This \(H^3\) statement is only a convenient nonoptimal extension. No claim is made concerning the minimal regularity or the optimal coercive topology.

\section*{Statements and Declarations}

\subsection*{Funding}

This work was supported by the Japan Society for the Promotion of Science (JSPS) KAKENHI Grant Number JP23KJ2232.

\subsection*{Competing interests}

The author has no relevant financial or non-financial interests to disclose.

\subsection*{Data availability}

No datasets were generated or analysed during the current study.

\subsection*{AI use declaration}

Generative AI tools, including ChatGPT (OpenAI), were used to assist with additional checks of calculations and logical consistency, literature searches, and the editing and organization of the manuscript. The research questions, mathematical arguments, calculations, interpretations, and conclusions were developed and verified by the author. All AI-assisted outputs used in the manuscript were independently checked by the author, who takes full responsibility for the content.

\printbibliography

\end{document}